\documentclass[a4paper, 12pt]{article}

\usepackage[english]{babel}
\usepackage[utf8]{inputenc}
\usepackage[T1]{fontenc}
\usepackage{amssymb,amsfonts,amsmath,amsthm}

\usepackage{graphicx}
\usepackage{multirow}
\usepackage{tikz}
\usetikzlibrary{arrows.meta, decorations.pathmorphing}

\usepackage[top= 3cm, bottom= 3cm, right= 3cm, left=3cm ]{geometry}
\usepackage{authblk}

\usetikzlibrary{arrows.meta}
\usepackage[]{enumerate}
\newcounter{pequation} % compteur spécial pour ce bloc

\usepackage{comment}
\usepackage{algorithm}
\usepackage{algorithm}
\usepackage{algpseudocode}

\newtheorem{theorem}{Theorem}

\newtheorem{lemma}{Lemma}
\newtheorem{remark}{Remark}
\newtheorem{corollary}{Corollary}

\newtheorem{proposition}{Proposition}

\numberwithin{equation}{section}
\numberwithin{definition}{section}
\numberwithin{theorem}{section}
\numberwithin{lemma}{section}
\numberwithin{remark}{section}
\numberwithin{corollary}{section}
\numberwithin{proposition}{section}
\newcommand{\R}{\mathbb{R}} %% reals
\usepackage{hyperref}

\begin{document}
\title{1D Convection-Diffusion in Porous Media: A Unified Solution from Bounded Domain to Half-Space}

%\author[1,2]{Koffi Ognandon AYENA}
%\author[2]{Frédéric HOLWECK}
%\author[3]{Amah Séna D'ALMEIDA}

%\affil[1]{Department of Mathematics, University of Lomé, Togo}
%\affil[2]{ICB/UTBM, UMR6303 CNRS, University of Technology of Belfort-Montbéliard, France}
%\affil[3]{Department of Mathematics, University of Lomé, Togo}

\author{
  Koffi Ognandon AYENA\thanks{Department of Mathematics, University of Lomé, Togo,\\ \href{mailto:kayena257@gmail.com}{kayena257@gmail.com}} ,
  %Frédéric HOLWECK\thanks{ICB/UTBM, UMR6303 CNRS, University of Technology of Belfort-Montbéliard, France} ,
  Amah Séna D'ALMEIDA\thanks{Department of Mathematics, University of Lomé, Togo}
}

\date{\today}
\maketitle

\begin{abstract}
	\noindent We present an analytical solution for 1D atmospheric pollutant dispersion in a bounded domain. Using coordinate transformation and spectral decomposition, we obtain a Fourier series solution for time-dependent wind velocity. For constant wind, Laplace inversion yields dual representations via residue calculus and Poisson summation. The solution extends to semi-infinite domains, recovering classical error-function profiles.
\end{abstract}

\section{Introduction}
Modelling the atmospheric dispersion of pollutants, such as radionuclides or fine particulate matter, is central to environmental and public-health risk assessment. This problem is traditionally described by a convection-diffusion equation, which combines transport by the wind (convection) with turbulent diffusion.

\begin{equation}
	\label{eq1}
	\frac{\partial}{\partial t} c(\overrightarrow{x}, t)+ \overrightarrow{\nabla}c(\overrightarrow{x},t)\cdot \overrightarrow{V}(\overrightarrow{x}, t)+c(\overrightarrow{x}, t) \overrightarrow{\nabla} \cdot\overrightarrow{V}(\overrightarrow{x}, t)  - \overrightarrow{\nabla} \cdot\Big(\mathbb{D}\cdot\overrightarrow{\nabla}c(\overrightarrow{x},t)\Big) = S(\overrightarrow{x}, t)
\end{equation}
where $\mathbb{D}$ is the diffusion tensor.\\
The analytical resolution of equation~\eqref{eq1}, in some specific cases, has been addressed by several authors using analytical or stochastic methods.

Historically, this equation has been studied extensively in the infinite-domain case ($\R^3$), where explicit analytical solutions can be obtained under simplifying assumptions. Among the classical approaches, the Gaussian plume solution allows one to solve the equation for a continuous point source in an infinite, homogeneous medium, producing Gaussian concentration distributions \cite{stockie2011mathematics}.
Other authors derive the steady-state solution using the method of characteristics \cite{bermudez1987methode}. The asymptotic nature of solutions to steady convection-diffusion problems has also been studied \cite{stynes2005steady}. Numerical methods for singular-perturbation problems, in particular convection-dominated steady convection-diffusion problems, have been widely discussed \cite{linss2003layer, augustin2011assessment}.
Zoppou and Knight (1999) proposed analytical solutions of the advection-diffusion equation with spatially varying coefficients in one to three dimensions, under the assumption that velocity varies proportionally to distance and diffusion scales with the square of this velocity \cite{zoppou1999analytical}. Separately, Szymczak and Ladd (2003) explored stochastic solutions, representing concentration through random-walk particles and proposing methods to impose reflecting, absorbing, or finite-reservoir boundary conditions \cite{szymczak2003boundary}. This approach reproduces boundary effects faithfully and offers a precise alternative to deterministic methods, particularly in complex geometries or near interfaces.

% NOTE TO THE AUTHOR: the literature review above stops in 2011. Adding a
% few references published after 2015 on analytical/numerical
% convection-diffusion would help position the contribution relative to
% the current state of the art.

However, for numerical, analytical or semi-analytical solution for concentration-driven flow in porous media, classical or frational approaches often treat bounded \cite{van1982analytical, deng2014integral, singh2015scale, fahs2014new} and unbounded~\cite{kunasegaran2025analytical} domains separately.%, requiring dedicated developments and obscuring the underlying unity of the solutions. 

We therefore present an approach that unifies the treatment of the longitudinal dispersion equation in a porous medium  whether posed on a bounded interval \( [0,L] \) or on the half-space \( [0, +\infty[ \).

This manuscript is organized as follows. In section~\ref{sec2}, we state the 
assumptions and formulate the problem. The derivation of the Laplace Transform solution and the Green's function via residue, with 
numerical illustration, is performed in section~\ref{sec3}. We present in section~\ref{sec4} a finite-volume validation and comparison with Genuchten's numerical solution. In section~\ref{sec5} a derivation of the Green's function via Poisson summation allows to establish and justify the $L\to+\infty$ limit recovering the half-space solution. The conclusion is given in section~\ref{sec_conclusion}.

%This manuscript is organized as follows.In Section~\ref{sec2}, we state the assumptions needed in the sequel and establish the equivalence between the spectral and spatial representations of the solution, showing explicitly how the limit \( L \to +\infty \) recovers the classical half-space solution in the constant-velocity case.

\section{Problem statement}
\label{sec2}
This section presents an analytical solution for the longitudinal dispersion equation in a bounded porous medium. The problem models the transport of a solute in a unidirectional flow with constant mean velocity.

\begin{equation}
	\frac{\partial}{\partial t} c(x, t)  +v_0\frac{\partial c(x, t)}{\partial x}
	- D_{11}\frac{\partial^2c(x, t)}{\partial x^2}   
	= 0%\rho_p\delta_{x_0},
	\label{eq:edp_dim1D}
\end{equation}
\begin{align}
	\label{eq:CIinf}
	c(x,0)&=c_a(x), \qquad \forall x \in ]0, L],\\
	\label{eq:CLinf}
	c(0,t)&=c_p(t) \qquad \forall t \in \R_+^*.%,\\
	%c(0,0)&=c_p^0, \\
	%c(L,0)&=c_a(L), 
\end{align}

\subsection{Assumptions}
\label{sec:hyp}
The results of the following sections rely on the assumptions below.

\begin{itemize}
	\item \label{H1} The physical coefficients are constant, with $D_{11}>0$ and $v_0\in\R$ (constant convection velocity).
	\item \label{H2} The initial condition is spatially uniform on the domain considered: $c_a(x)=c_a^0$ for all $x$, where $c_a^0\in\R$ is a given constant. (This assumption is used from Section~\ref{sec3} onward, where the Laplace transform of $c_a$ is simplified accordingly; the case of a non-constant $c_a(x)$ remains covered by the general solution~\eqref{sol_general}.)
	\item \label{H3} The boundary condition $c_p$ is of class $\mathcal{C}^1(\R_+)$.
\end{itemize}

\section{Solution via the Laplace transform}
\label{sec3}
The Laplace transform of a function $f(t)$ is defined by

\[
F(s) = \mathcal{L}\Big[f(t)\Big] = \int_{0}^{\infty} e^{-st} f(t) \, dt
\]

where $s$ is a complex variable.
Let $C(x,s)=\mathcal{L}\Big[c(x,t)\Big]$ denote the Laplace transform of $c(x,t)$.

Applying the operator $\mathcal{L}$ to equation~\eqref{eq:edp_dim1D} gives

\begin{align}
	\label{laplace}
	sC(x,s)+v_0\frac{\partial C(x, s)}{\partial x}-D_{11}\frac{\partial^2C(x, s)}{\partial x^2}=c_a(x)%\frac{\rho_p}{s}\delta_{x_0}+c_a(x)
\end{align}

Set $y(x)=C(x,s)$. A homogeneous solution of equation \eqref{laplace} is

\begin{equation}
	C_h(x,s)=K_1e^{\lambda_1x}+K_2e^{\lambda_2x},
\end{equation}
with
$$
\lambda_1=\frac{v_0+\sqrt{v_0^2+4sD_{11}}}{2D_{11}}, \qquad  \lambda_2=\frac{v_0-\sqrt{v_0^2+4sD_{11}}}{2D_{11}}, \qquad \Delta = v_0^2+4sD_{11}\ge 0
$$
for $s\geq 0$. 
A particular solution of \eqref{laplace} is
\begin{equation}
	\label{sol_particular}
	C_p(x,s)=k_1(x)e^{\lambda_1x}+k_2(x)e^{\lambda_2x}
\end{equation}
with coefficients $k_1$ and $k_2$ given by
$$
k_1(x)=\frac{1}{D_{11}(\lambda_2-\lambda_1)}\int \frac{c_a(x)}{e^{\lambda_1x}}dx\qquad k_2(x)=\frac{1}{D_{11}(\lambda_1-\lambda_2)}\int \frac{c_a(x)}{e^{\lambda_2x}}dx
$$
Every solution of \eqref{laplace} is therefore of the form

\begin{multline}
	\label{sol_general}
	C(x,s)=K_1e^{\lambda_1x}+K_2e^{\lambda_2x}+\\
	\frac{e^{\lambda_1x}}{D_{11}(\lambda_2-\lambda_1)}\int \frac{c_a(x)}{e^{\lambda_1x}}dx+\frac{e^{\lambda_2x}}{D_{11}(\lambda_1-\lambda_2)}\int \frac{c_a(x)}{e^{\lambda_2x}}dx
\end{multline}
with $\Delta \ge 0$.
When $c_a$ is constant, $c_a(x)=c_a^0$ (Assumption~\ref{H2}), and setting
$$
\lambda_1=\alpha +\beta,\qquad \lambda_2=\alpha -\beta 
, \quad 
\alpha =\frac{v_0}{2D_{11}}, \quad \beta=\sqrt{\frac{s}{D_{11}}-s_0},\qquad s_0=-\frac{v_0^2}{4D_{11}^2}
$$
and noting that
$$
\lambda_1\lambda_2=-\frac{s}{D_{11}},
$$
the solution $C(x,s)$ takes the form
\begin{align}
	\label{e10}
	C(x,s)
	&=K_1e^{\lambda_1x}+K_2e^{\lambda_2x}
	+\frac{c_a^0}{s}%+\frac{\rho_p}{sD_{11}(\lambda_1-\lambda_2)} \Bigg(e^{\lambda_2(x-x_0)}-e^{\lambda_1(x-x_0)}\Bigg)
\end{align}
The constants $K_i, i=1,2$ are determined explicitly using the boundary conditions
\begin{align}
	C(L,0)&=\frac{c_a(L)}{s}=\frac{c_a^0}{s}  ,\\
	\mathcal{L}[c(0,t)]&=C_p(s) \qquad \forall t \in \R_+^*,
\end{align}

Consequently, \eqref{e10} becomes
\begin{align}
	\label{sol_final}
	C(x,s)= \frac{e^{\lambda_2L}e^{\lambda_1x}-e^{\lambda_1L}e^{\lambda_2x}}{e^{\lambda_1L}-e^{\lambda_2L}}\left(\frac{c_a^0}{s}-C_p(s)\right)
	+\frac{c_a^0}{s}
\end{align}
The sought solution $c(x,t)$ is, by definition, the inverse Laplace transform of $C(x,s)$ given by~\eqref{sol_final}.
For this reason, we rewrite $C(x,s)$ as
\begin{align}
	C(x,s)&= G(x,s)\left(\frac{c_a^0}{s}-C_p(s)\right)
	+\frac{c_a^0}{s},\qquad G(x,s)=-e^{\alpha x}\frac{\sinh ((L-x)\beta)}{\sinh (L\beta)}
\end{align}
Applying $\mathcal{L}^{-1}$,
\begin{align*}
	c(x,t)=c_a^0+\mathcal{L}^{-1}\Big[G(x,s)\Big]*\mathcal{L}^{-1}\Big[\frac{c_a^0}{s}-C_p(s)\Big]
\end{align*}
and using the property of the convolution product, denoted $*$, the solution of equation~\eqref{eq:edp_dim1D} becomes
\begin{align*}
	c(x,t)=c_a^0+\int_{0}^{t}g(x,\tau)\left(c_a^0-c_p(t-\tau)\right)d\tau.
\end{align*}
\subsection{Inversion of $G$ by the residue method}
The inverse Laplace transform of $G$ is computed through a contour integral in the complex plane. The Bromwich-Mellin formula reads
\begin{equation}
	g(x,t)=\frac{1}{2i\pi}\int_{\gamma-i\infty}^{\gamma+i\infty}G(x,s)e^{st}ds
\end{equation}
We invert $G(x,s)$ using residues.
The singularities of $G(x,s)$ come from the zeros of its denominator:

\begin{align}
	\label{pole1}
	\sinh(\beta L)=0&\implies \beta L=in\pi\qquad(n\in \mathbb{Z}^*)\\
	\label{pole2}
	&\implies s_n = -\frac{v_0^2}{4D_{11}}-\frac{n^2\pi^2D_{11}}{L^2}
\end{align}
Since the poles $s_n$ are strictly negative, we choose $\gamma$ such that its real part $\mathfrak{R}(\gamma)\in \mathbb{R}^*_+$. Since $s_n=s_{-n}$, the residue theorem allows us to rewrite the inverse Laplace transform as

\begin{align}
	g(x,t)=\frac{1}{2}\sum_{n\in \mathbb{Z}^*}Res(G(x,s)e^{st},s_n).
\end{align}

Since the poles $s_n$ are of order $1$, the residue computation gives
\begin{align*}
	Res(G(x,s)e^{st},s_n)&=\lim_{s\to s_n}(s-s_n)G(x,s)e^{st}
\end{align*}
Since the denominator $\sinh (\beta L)$ of $G(x,s)$ vanishes at $s_n$, its first-order Taylor expansion near $s_n$ is
$$
\sinh (\beta L)\approx \sinh (\beta L)\Bigg|_{s=s_n}+(s-s_n)\frac{d}{ds}\sinh (\beta L)\Bigg|_{s=s_n}\approx  (s-s_n)\frac{d}{ds}\sinh (\beta L)\Bigg|_{s=s_n}
$$
so the residue becomes
\begin{align*}
	Res(G(x,s)e^{st},s_n)&=-e^{\alpha x+s_nt}\lim_{s\to s_n}(s-s_n)\frac{\sinh ((L-x)\beta)}{(s-s_n)\frac{d}{ds}\sinh (L\beta)}
\end{align*}
From
$$
\frac{d}{ds}\sinh (\beta z)=z\cosh(\beta z) \frac{d\beta}{ds}=\frac{z\cosh(\beta z)}{\sqrt{v_0^2+4sD_{11}}}, 
$$
replacing $z$ by $L$ and using implication~\eqref{pole1}, we obtain
$$
\frac{d}{ds}\sinh (\beta z)\Bigg|_{s=s_n}=\frac{(-1)^nL^2}{2in\pi D_{11}},
$$
with $\sqrt{v_0^2+4s_nD_{11}}=\frac{2in\pi D_{11}}{L}$.

Using the property $\sinh(ix)=i\sin x$ and
$$\beta(L-x)\Bigg|_{s=s_n}=\beta L(1-\frac{x}{L})\Bigg|_{s=s_n}=in\pi(1-\frac{x}{L}),$$
we simplify the residue expression:

\begin{align}
	\nonumber
	Res(G(x,s)e^{st},s_n)&=-\frac{i\sin (n\pi(1-\frac{x}{L}))}{\frac{(-1)^nL^2}{2in\pi D_{11}}}e^{\alpha x+s_nt}\\
	\nonumber
	&=-\frac{(-1)^n2n\pi D_{11}\sin (\frac{n\pi x}{L})}{(-1)^nL^2}e^{\alpha x+s_nt}\\
	&=-\frac{2n\pi D_{11}}{L^2}e^{\alpha x+s_nt}\sin (\frac{n\pi x}{L})
\end{align}
Using the symmetry $s_n=s_{-n}$ and $\sin (\frac{-n\pi x}{L})=-\sin (\frac{n\pi x}{L})$, we group the $+n$ and $-n$ terms to obtain
\begin{align*}
	g(x,t)&=\frac{1}{2}\sum_{n\in \mathbb{Z}^*}Res(G(x,s)e^{st},s_n)\\
	&=-\frac{1}{2}e^{\alpha x}\sum_{n\in \mathbb{Z}^*}\frac{2n\pi D_{11}}{L^2}e^{s_nt}\sin (\frac{n\pi x}{L})
\end{align*}
and consequently
\begin{align}
	\label{residu_g}
	g(x,t)=-e^{\alpha x}\sum_{n=1}^{\infty}\frac{2n\pi D_{11}}{L^2}e^{s_nt}\sin (\frac{n\pi x}{L})
\end{align}
We obtain the analytical expression of the general solution from the known function $g$:
\begin{align}
	\label{solution}
	c(x,t)
	&=c_a^0+\int_{0}^{t}g(x,\tau)\left(c_a^0-c_p(t-\tau)\right)d\tau
\end{align}
The integral
$$
\int_{0}^{t}g(x,\tau)\left(c_a^0-c_p(t-\tau)\right)d\tau
$$
is evaluated numerically below using a deterministic quadrature rule.

\begin{remark}[Recovery of the boundary condition at $x=0$]
	\label{rem:boundary_x0}
	For every $n\geq1$ and every $\tau>0$, $\sin(n\pi\cdot 0/L)=0$, so the series
	\eqref{residu_g} gives $g(0,\tau)=0$ for all $\tau>0$. Substituting this into
	\eqref{solution} would therefore yield $c(0,t)=c_a^0$ for every $t$ (figure~\ref{fig:concentration_1D}), in apparent
	contradiction with the boundary condition~\eqref{eq:CLinf}, $c(0,t)=c_p(t)$.
	
	This is not a contradiction: the Green's function $g(0,\cdot)$ is not a
	classical function on $[0,+\infty[$ but a distribution. Indeed, in the Laplace
	domain, $$G(0,s)=-\sinh(L\beta)/\sinh(L\beta)=-1$$
    for every $s$, independently
	of $D_{11}$ and $v_0$; hence
	$$
	g(0,\cdot)=\mathcal{L}^{-1}[-1]=-\delta_0
	$$
	in the sense of distributions, where $\delta_0$ is the Dirac mass at $\tau=0$.
	The series~\eqref{residu_g}, obtained from the poles of $G(x,s)$ alone,
	only reproduces the regular part of $g(0,\cdot)$ on $\tau>0$ -- correctly equal
	to $0$ -- but does not capture this singular contribution at $\tau=0$.
	Formally injecting $g(0,\tau)=-\delta_0(\tau)$ into~\eqref{solution} gives
	$$
	c(0,t)=c_a^0+\int_0^t\big(-\delta_0(\tau)\big)\big(c_a^0-c_p(t-\tau)\big)\,d\tau
	=c_a^0-\big(c_a^0-c_p(t)\big)=c_p(t),
	$$
	consistent with condition~\eqref{eq:CLinf}. Equivalently, the boundary value is recovered
	as the limit $\lim_{x\to0^+}c(x,t)=c_p(t)$ (a boundary layer at $x=0$), rather
	than by direct substitution of $x=0$ into the integral formula valid for $t>0$.
	
	Numerically, this means the convolution integral in equation~\eqref{solution} must not
	be evaluated at $x=0$: the value $c(0,t)=c_p(t)$ has to be imposed directly, as
	is done in the implementation (figure~\ref{fig:concentration_1D_ok}). By contrast, no such difficulty arises at $x=L$,
	since $\sin(n\pi)=0$ makes $g(L,\tau)=0$ classically (no Dirac contribution),
	consistently with $G(L,s)=0$ for every $s$, so $c(L,t)=c_a^0$ is obtained
	directly from~\eqref{solution} without special treatment.
\end{remark}
The physical parameters used are: diffusion coefficient $D_{11} = 0.1$, convection velocity $v_0 = 0.5$, uniform initial concentration $c_a^0 = 0.0$, and boundary condition at $x = 0$ given by
\[
c_p(t) = 
\begin{cases} 
	5\,t & \text{if } t < 0.1 \\
	0.5 & \text{if } t \geq 0.1
\end{cases}
\]
(a linear ramp followed by a plateau). The zero-flux condition at $x = L$ is implicitly satisfied by the construction of the Fourier series.

\textbf{Numerical method:} The results shown below are obtained via formula~\eqref{solution}
where $g(x,t)$ is the Green's function developed as a Fourier series (equation \eqref{residu_g}).
The convolution integral is evaluated numerically using the \texttt{quad} routine from \textsc{SciPy}.

\textbf{Stability and convergence:} The series is truncated at $N = 500$ terms, which gives good numerical stability. The choice of $N$ is guided by the following convergence criterion: the terms decay exponentially as $\exp(-n^2\pi^2 D_{11} t / L^2)$ for $t > 0$, so the series converges absolutely for every $t > 0$. At short times ($t \to 0^+$), convergence is slower, which justifies using a fairly large number of terms ($N = 500$) to capture the concentration front accurately.
A comparison with a finite-volume scheme showed results consistent with the analytical solution.
% NOTE TO THE AUTHOR: reporting a quantitative error here (e.g. an L2/L_inf
% error table as a function of h and Delta t) would strengthen this
% comparison; a qualitative statement alone is unlikely to be sufficient
% for a journal expecting quantitative numerical evidence.

Profiles are shown at times $t =0.02, 0.05, 0.1, 0.2$, $0.5$, $1.0$ and $2.0$. The plots show the convective transport of the front to the right, its diffusive broadening, and a gradual build-up of material near the wall $x = L$ associated with the zero-flux condition.

\begin{figure}[htb]
	\centering
	\includegraphics[width=\linewidth, height=7cm, keepaspectratio]{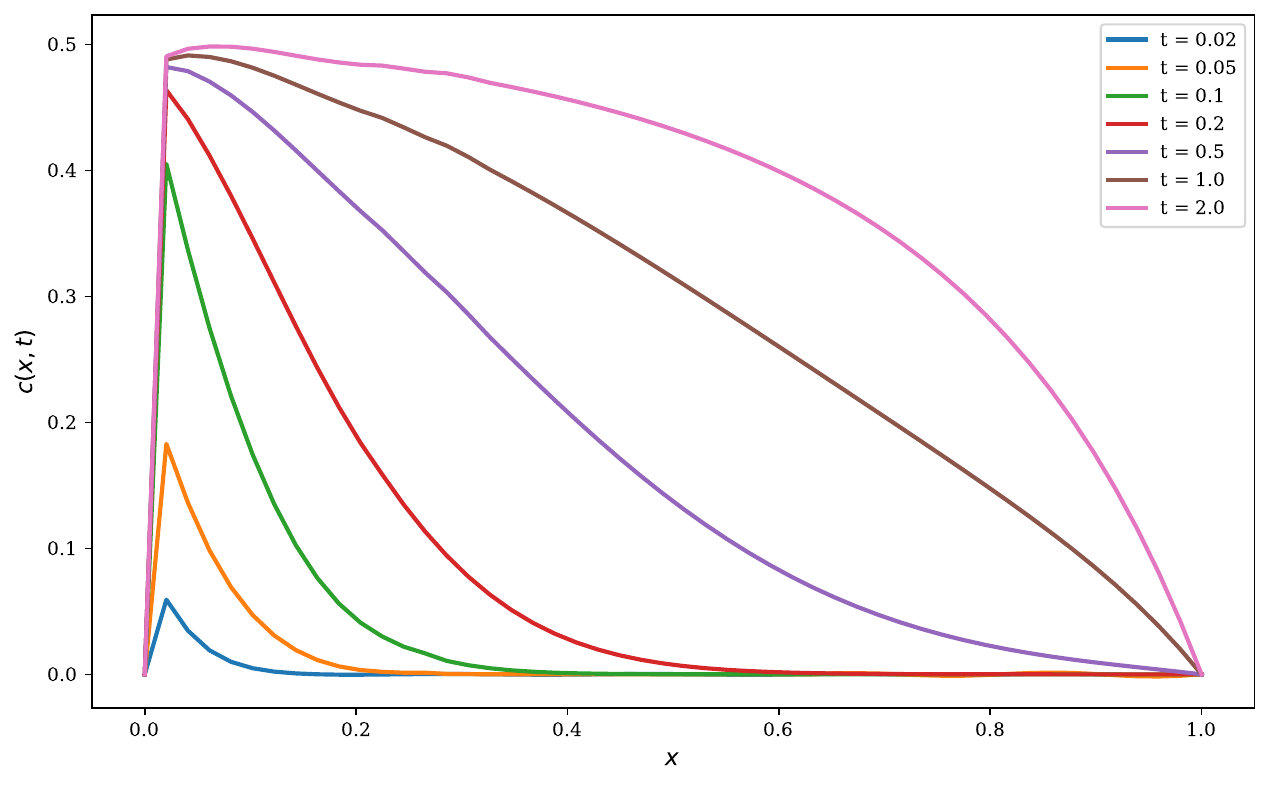}
\caption{
    Concentration profiles $c(x,t)$ obtained by direct substitution of the
	Fourier-series Green's function~\eqref{residu_g} into~\eqref{solution},
	without accounting for the Dirac contribution at $x=0$ identified in
	Remark~\ref{rem:boundary_x0}.}
	\label{fig:concentration_1D}
\end{figure}

\begin{figure}[htb]
	\centering
	\includegraphics[width=\linewidth, height=7cm, keepaspectratio]{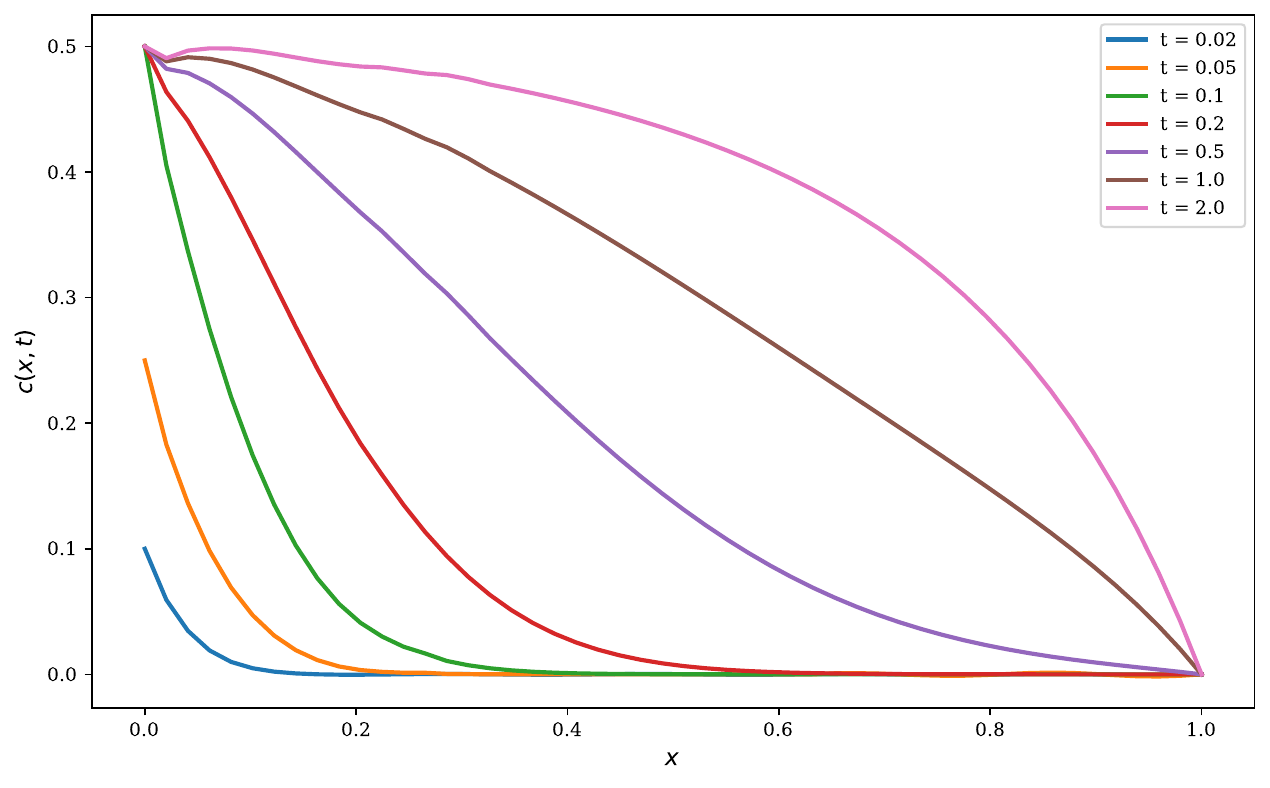}
\caption{
	Concentration profiles $c(x,t)$ obtained from the same analytical
	solution~\eqref{residu_g}--\eqref{solution}, this time accounting for the
	Dirac mass $g(0,\cdot)=-\delta_0$ identified in
	Remark~\ref{rem:boundary_x0}: the boundary value $c(0,t)=c_p(t)$ is
	imposed directly rather than obtained from the series at $\tau>0$ alone.
	The profile is now continuous at $x=0$ and satisfies the prescribed
	boundary condition~\eqref{eq:CLinf}.}
    
    %Concentration profiles $c(x,t)$ on the spatial domain $x \in [0, L]$ with $L = 1.0$, obtained from the analytical Fourier-series solution. Here the Dirichlet condition $c(0,t)=c_p(t)$ is enforced directly through the boundary flux~\eqref{eq:flux_gauche} in the discretization, with no distributional correction needed.}
	\label{fig:concentration_1D_ok}
\end{figure}

\section{Numerical results}
\label{sec4}
\subsection{Finite-volume discretization}
\label{sub:vf}

We return to the 1D convection-diffusion problem~\eqref{eq:edp_dim1D}, written in conservative form:
\[
\partial_t c + \partial_x F(c) = 0, \qquad
F(x,t) := v_0\,c(x,t) - D_{11}\,\partial_x c(x,t),
\]
where $F$ is the total flux (advective + diffusive). This form will be referred to as $\eqref{eq:edp}$ below.
\begin{equation}
	\label{eq:edp}
	\partial_t c + \partial_x F(c) = 0.
\end{equation}

\subsubsection{Semi-discretization in space}

We introduce a uniform subdivision of $[0,L]$ into $N$ control
volumes
\[
\Omega_i = \left[x_{i-\frac12},\, x_{i+\frac12}\right], \qquad
x_{i+\frac12} = ih, \qquad h = \frac{L}{N}, \qquad i=1,\dots,N,
\]
with centers $x_i = \left(i-\tfrac12\right)h$. We denote by $c_i(t)$ the approximation
of the average value of $c(\cdot,t)$ over $\Omega_i$.

Integrating the relation~\eqref{eq:edp} over
$\Omega_i$ and dividing by $h$ gives the semi-discrete scheme
(method of lines)
\begin{equation}
	\label{eq:semi_discret}
	\frac{dc_i}{dt} = -\frac{1}{h}\Big(F_{i+\frac12} - F_{i-\frac12}\Big) .
\end{equation}

It remains to specify the discretization of the flux $F$ at each interface. For internal fluxes, we use an upwind scheme ($v_0>0$) for the advective term and a centred scheme for the diffusive term:
\begin{equation}
	\label{eq:flux_interne}
	F_{i+\frac12} = v_0 c_i - D_{11}\frac{c_{i+1}-c_i}{h}, \qquad i=1,\dots,N-1.
\end{equation}
At the left boundary ($x=0$), the Dirichlet condition \eqref{eq:CLinf} imposes $c(0,t)=c_p(t)$; approximating the derivative with a finite difference over the half-step separating $x=0$ from the centre $x_1$ of the first cell gives
\begin{equation}
	\label{eq:flux_gauche}
	F_{\frac12} = v_0 c_p(t) - D_{11}\frac{c_1 - c_p(t)}{h/2}.
\end{equation}
At the right boundary ($x=L$), the condition $c(L,0)=c_a^0$ is treated analogously:
\begin{equation}
	\label{eq:flux_droit}
	F_{N+\frac12} = v_0 c_a^0 - D_{11}\frac{c_a^0 - c_N}{h/2}.
\end{equation}

We detail here, step by step, the construction of the full matrix form of the semi-discretized system.

\subsection*{Semi-discrete equation for each volume}

Equation \eqref{eq:semi_discret} reads, for each $i=1,\dots,N$:
\begin{equation}
	\frac{dc_i}{dt} = -\frac{1}{h}\left(F_{i+\frac12} - F_{i-\frac12}\right).
\end{equation}

\subsection*{Expression of the fluxes in terms of the unknowns}

\subsubsection*{Internal fluxes ($i=1,\dots,N-1$)}

From \eqref{eq:flux_interne}:
\begin{equation}
	F_{i+\frac12} = v_0 c_i - D_{11}\frac{c_{i+1}-c_i}{h}
	= \left(v_0 + \frac{D_{11}}{h}\right)c_i - \frac{D_{11}}{h} c_{i+1}.
\end{equation}

\subsubsection*{Flux at the left boundary ($F_{\frac12}$)}

From \eqref{eq:flux_gauche}:
\begin{equation}
	F_{\frac12} = v_0 c_p(t) - D_{11}\frac{c_1 - c_p(t)}{h/2}
	= v_0 c_p(t) - \frac{2D_{11}}{h}(c_1 - c_p(t)).
\end{equation}

Expanding:
\begin{equation}
	F_{\frac12} = -\frac{2D_{11}}{h} c_1 + \left(v_0 + \frac{2D_{11}}{h}\right) c_p(t).
\end{equation}

\subsubsection*{Flux at the right boundary ($F_{N+\frac12}$)}

From \eqref{eq:flux_droit}:
\begin{equation}
	F_{N+\frac12} = v_0 c_a^0 - D_{11}\frac{c_a^0 - c_N}{h/2}
	= v_0 c_a^0 - \frac{2D_{11}}{h}(c_a^0 - c_N).
\end{equation}

Expanding:
\begin{equation}
	F_{N+\frac12} = \frac{2D_{11}}{h} c_N + \left(v_0 - \frac{2D_{11}}{h}\right) c_a^0.
\end{equation}

\subsection*{Equation for each index $i$}

\subsubsection*{Case $i=1$ (first cell)}

For $i=1$, the equation involves $F_{\frac12}$ (left boundary) and $F_{\frac32}$ (internal flux between cells 1 and 2):
\begin{equation}
	\frac{dc_1}{dt} = -\frac{1}{h}\left(F_{\frac32} - F_{\frac12}\right) .
\end{equation}

The equation for $i=1$ becomes
\begin{equation}
	\frac{dc_1}{dt} = -\left(\frac{v_0}{h} + \frac{3D_{11}}{h^2}\right)c_1 + \frac{D_{11}}{h^2} c_2 + \left(\frac{v_0}{h} + \frac{2D_{11}}{h^2}\right)c_p(t) + b_1.
\end{equation}

\subsubsection*{Case $2 \le i \le N-1$ (internal cells)}

The equation for internal $i$ becomes
\begin{equation}
	\frac{dc_i}{dt} = \left(\frac{v_0}{h} + \frac{D_{11}}{h^2}\right)c_{i-1} - \left(\frac{v_0}{h} + \frac{2D_{11}}{h^2}\right)c_i + \frac{D_{11}}{h^2} c_{i+1} .
\end{equation}

\subsubsection*{3.3 Case $i=N$ (last cell)}

For $i=N$, the equation involves $F_{N-\frac12}$ (internal flux between $N-1$ and $N$) and $F_{N+\frac12}$ (right boundary):

The equation for $i=N$ becomes
\begin{equation}
	\frac{dc_N}{dt} = \left(\frac{v_0}{h} + \frac{D_{11}}{h^2}\right)c_{N-1} - \frac{3D_{11}}{h^2} c_N - \left(\frac{v_0}{h} - \frac{2D_{11}}{h^2}\right) c_a^0 + b_N.
\end{equation}

The system can now be written as
\begin{equation}
	\frac{d\mathbf{c}}{dt} = M\mathbf{c} + c_p(t)\mathbf{g} + \mathbf{s},
\end{equation}
where $\mathbf{c} = (c_1,\dots,c_N)^\top$.

Reading off the coefficients $c_j$ in each equation gives a tridiagonal matrix:

\begin{itemize}
	\item \textbf{Row 1}:
	\begin{equation}
		M_{1,1} = -\left(\frac{v_0}{h} + \frac{3D_{11}}{h^2}\right), \qquad M_{1,2} = \frac{D_{11}}{h^2}.
	\end{equation}
	
	\item \textbf{Rows $i=2,\dots,N-1$}:
	\begin{equation}
		M_{i,i-1} = \frac{v_0}{h} + \frac{D_{11}}{h^2}, \qquad
		M_{i,i} = -\left(\frac{v_0}{h} + \frac{2D_{11}}{h^2}\right), \qquad
		M_{i,i+1} = \frac{D_{11}}{h^2}.
	\end{equation}
	
	\item \textbf{Row $N$}:
	\begin{equation}
		M_{N,N-1} = \frac{v_0}{h} + \frac{D_{11}}{h^2}, \qquad
		M_{N,N} = -\frac{3D_{11}}{h^2}.
	\end{equation}
\end{itemize}

All other coefficients are zero.

Only the first row contains a term in $c_p(t)$:
\begin{equation}
	g_1 = \frac{v_0}{h} + \frac{2D_{11}}{h^2}, \qquad g_i = 0 \quad \text{for } i\ge 2.
\end{equation}

Hence
\begin{equation}
	\mathbf{g} = \left(\frac{v_0}{h} + \frac{2D_{11}}{h^2},\, 0,\, \dots,\, 0\right)^\top.
\end{equation}

The vector $\mathbf{s}$ contains two types of contributions:
The right Dirichlet condition, which appears only in the last row as
	\begin{equation}
		-\left(\frac{v_0}{h} - \frac{2D_{11}}{h^2}\right) c_a^0.
	\end{equation}

Hence
\begin{equation}
	s_i = 0%\frac{\rho_p}{h}\,\mathbf{1}_{\{i=i_0\}}
    \quad \text{for } i=1,\dots,N-1,
\end{equation}
and
\begin{equation}
	s_N = - \left(\frac{v_0}{h} - \frac{2D_{11}}{h^2}\right) c_a^0.
\end{equation}
In summary, the semi-discretized system reads
\begin{equation}
	\boxed{
		\frac{d\mathbf{c}}{dt} = M\mathbf{c} + c_p(t)\mathbf{g} + \mathbf{s}
	}
\end{equation}

with

\begin{equation}
	M = 
	\begin{pmatrix}
		-\left(\frac{v_0}{h} + \frac{3D_{11}}{h^2}\right) & \frac{D_{11}}{h^2} & 0 & \cdots & 0 \\
		\frac{v_0}{h} + \frac{D_{11}}{h^2} & -\left(\frac{v_0}{h} + \frac{2D_{11}}{h^2}\right) & \frac{D_{11}}{h^2} & \ddots & \vdots \\
		0 & \ddots & \ddots & \ddots & 0 \\
		\vdots & \ddots & \frac{v_0}{h} + \frac{D_{11}}{h^2} & -\left(\frac{v_0}{h} + \frac{2D_{11}}{h^2}\right) & \frac{D_{11}}{h^2} \\
		0 & \cdots & 0 & \frac{v_0}{h} + \frac{D_{11}}{h^2} & -\frac{3D_{11}}{h^2}
	\end{pmatrix},
\end{equation}

\begin{equation}
	\mathbf{g} = 
	\begin{pmatrix}
		\frac{v_0}{h} + \frac{2D_{11}}{h^2} \\
		0 \\
		\vdots \\
		0
	\end{pmatrix},
\end{equation}
and
\begin{equation}
	\mathbf{s} = 
	\begin{pmatrix}
		0\\
		0\\
		\vdots \\
		0 \\
		 - \left(\frac{v_0}{h} - \frac{2D_{11}}{h^2}\right) c_a^0
	\end{pmatrix}.
\end{equation}

Applying the implicit Euler scheme gives
\begin{equation}
	\frac{\mathbf{c}^{n+1}-\mathbf{c}^n}{\Delta t} = M\mathbf{c}^{n+1} + c_p(t^{n+1})\mathbf{g} + \mathbf{s},
\end{equation}
that is,
\begin{equation}
\label{sol_FV}
		(I_N - \Delta t M)\mathbf{c}^{n+1} = \mathbf{c}^n + \Delta t\left(c_p(t^{n+1})\mathbf{g} + \mathbf{s}\right).
\end{equation}

The matrix $A = I_N - \Delta t M$ is tridiagonal and constant, allowing a single $LU$ factorization and an efficient $O(N)$ solve at each time step (figure~\ref{fig:concentration_1D_VF}).

\begin{figure}[h!]
	\centering
	\includegraphics[width=\linewidth, height=7cm, keepaspectratio]{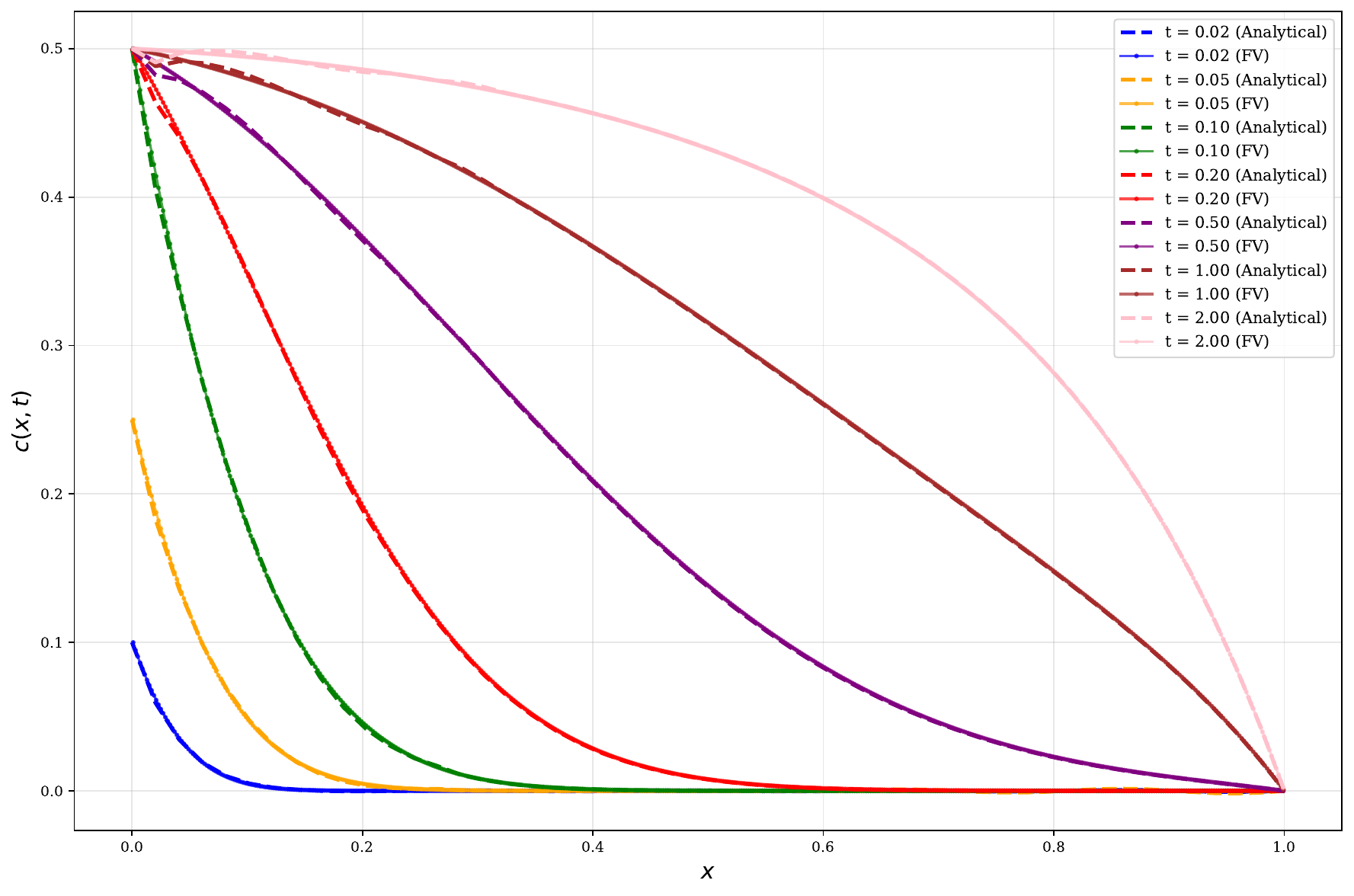}
    %comparison_solvers.png
    %profil_concentration_FV
	\caption{
		Concentration profiles $c(x,t)$ on the spatial domain $x \in [0, L]$ with $L = 1.0$, obtained with the finite-volume scheme.}
	\label{fig:concentration_1D_VF}
\end{figure}

\subsection{Comparison with the analytical solution of Genuchten}
\label{solGenuchten}
We compare our results with the analytical solution of Genuchten in~\cite{van1982analytical}, taking
$c_{a0} = 0$ (zero initial concentration), $L = 1.0$~km, $D_{11} = 0.1$~km$^2$/year,
$v_0 = 0.5$~km/year, and an inlet concentration $C_0 = 0.5$.
We observe very good agreement between the two solutions for $x$ below a certain
threshold, regardless of the time considered (figure~\ref{fig:concentration_1D_Genuchten}). However, as $t$ increases, a
discrepancy appears and progressively widens as we approach the right boundary
$x = L$.
This discrepancy is not due to numerical error, but rather to a different choice
of boundary condition at the outlet:

\begin{itemize}
    \item Authors impose in~\cite{van1982analytical} a homogeneous Neumann condition
    \[
        \left.\frac{\partial c}{\partial x}\right|_{x=L} = 0,
    \]
    representing a closed boundary with no outgoing dispersive flux. Under
    steady-state conditions, this forces the concentration to saturate over the
    entire domain, including at $x = L$, which explains the progressive
    accumulation observed in their profiles at large times.

    \item Our solution, built on an eigenbasis of $\sin(n\pi x/L)$,
    implicitly imposes a Dirichlet condition at $x = L$,
    \[
        c(L,t) = c_{a0},
    \]
    so the concentration remains anchored at $c_{a0}$. The resulting
    steady-state profile is therefore bounded between $c_p(t)$ at $x = 0$ and
    $c_{a}^0$ at $x = L$, and never fully saturates.
\end{itemize}

This difference in the treatment of the right boundary explains the growing
discrepancy observed in the plots at large times: it reflects a \emph{physical
difference in the modeling of the outlet condition}, rather than two solutions
of strictly the same problem. A rigorous comparison would require harmonizing
this condition -- either by imposing a zero-flux condition in our formulation;
or by introducing an explicit sink in Genuchten's solution.

\begin{figure}[h!]
	\centering
	\includegraphics[width=\linewidth, height=11cm, keepaspectratio]{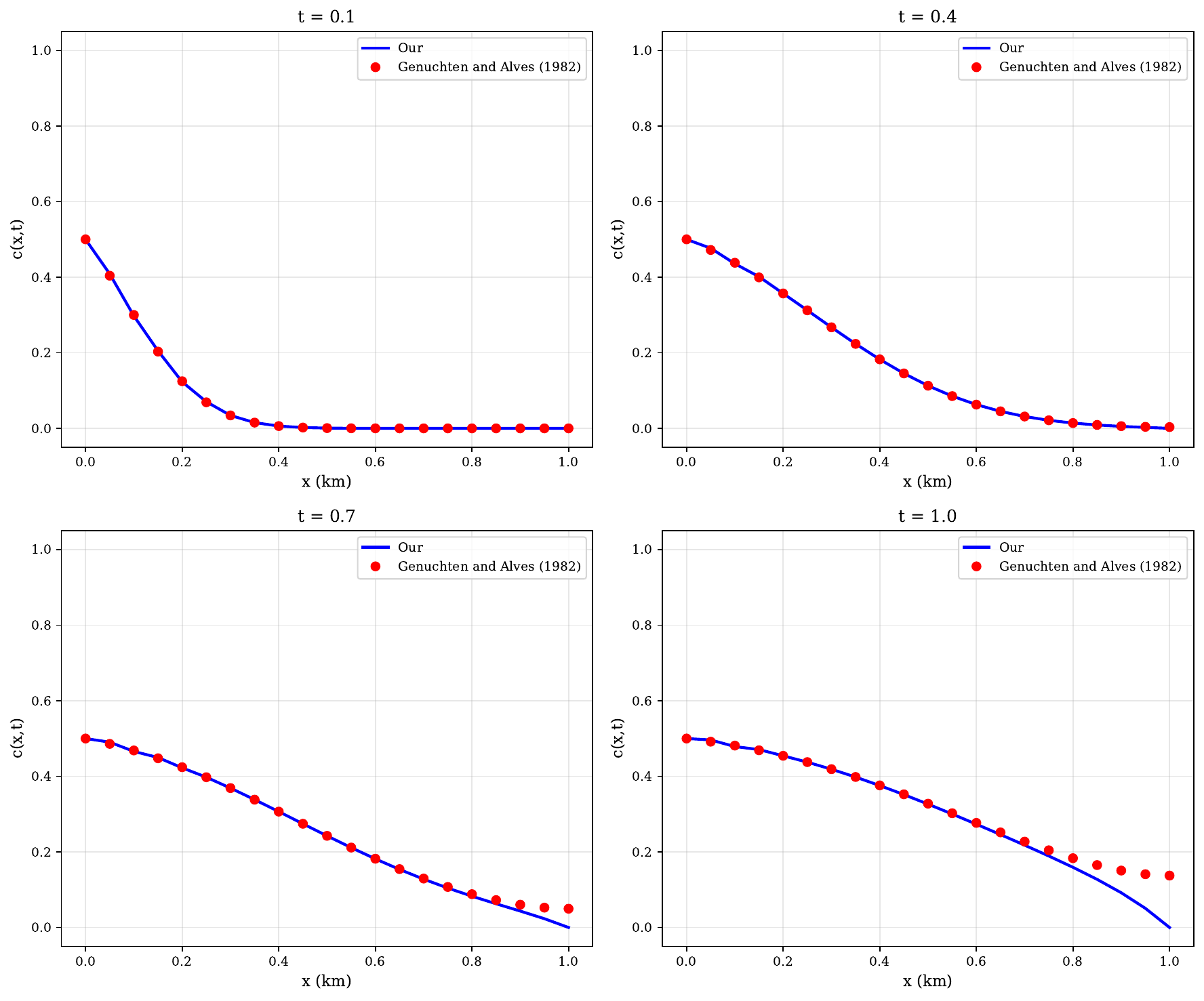}
    %comparison_solvers.png
    %profil_concentration_FV
	\caption{
		Concentration profiles $c(x,t)$ on the spatial domain $x \in [0, L]$ with $L = 1.0$, obtained with the analytical solution of Genuchten.}
	\label{fig:concentration_1D_Genuchten}
\end{figure}

Since the goal is to derive the analytical solution by extending the bounded domain $[0,L]$ to $[0,\infty[$, an equivalent form of the solution is needed. 

\section{Analytical solution in the half-space}
\label{sec5}
In this first part, we recover the same result \eqref{residu_g} by expanding $G(x,s)$ in series.

\subsection{Inversion of $G$ by the Poisson summation formula}
\begin{align*}
	G(x,s)&=-e^{\alpha x-\beta L}\frac{e^{(L-x)\beta}-e^{-(L-x)\beta}}{1-e^{-2\beta L}}\\
	&=-e^{\alpha x-\beta L}(e^{(L-x)\beta}-e^{-(L-x)\beta})\sum_{n=0}^{\infty} e^{-2n\beta L}, \qquad \text{since}\ |e^{-2\beta L}|< 1\\
	&=-e^{\alpha x}\sum_{n=0}^{\infty} e^{-\beta(x+2n L)}+e^{\alpha x}\sum_{n=0}^{\infty} e^{-\beta((2n+1) L-x)}
\end{align*}
We obtain the solution $c(x,t)$ by applying the inverse Laplace transform $\mathcal{L}^{-1}$ in reference~\cite{abramowitz1948handbook}:
\begin{align*}
	\mathcal{L}^{-1}[e^{-a\sqrt{s}}]= \frac{a}{2\sqrt{\pi t^3}}e^{-\frac{a^2}{4t}}, \qquad(a > 0) 
\end{align*}
and knowing by definition that
$$
\mathcal{L}^{-1}\Big[F(s-s_0)\Big]=e^{s_0t}\mathcal{L}^{-1}\Big[F(s)\Big], 
$$

we get
\begin{align}
	\mathcal{L}^{-1}[e^{-a\sqrt{\frac{s}{b}-s_0}}]= \frac{ae^{bs_0t}}{2\sqrt{\pi bt^3}}e^{-\frac{a^2}{4bt}}, \qquad \qquad(a > 0)
\end{align}
Applying this last relation to compute $g(x,t)$:
\begin{align*}
	g(x,t)&=\mathcal{L}^{-1}\Big[G(x,s)\Big]\\
	&=-e^{\alpha x}\sum_{n=0}^{\infty} \mathcal{L}^{-1}\Big[e^{-\beta(x+2n L)}\Big]+e^{\alpha x}\sum_{n=0}^{\infty} \mathcal{L}^{-1}\Big[e^{-\beta((2n+1) L-x)}\Big]\\
	&=-e^{\alpha x}\sum_{n=0}^{\infty} \frac{(x+2n L)e^{D_{11}s_0t}}{2\sqrt{\pi D_{11}t^3}}e^{-\frac{(x+2n L)^2}{4D_{11}t}}+e^{\alpha x}\sum_{n=0}^{\infty} \frac{((2n+1)L-x)e^{D_{11}s_0t}}{2\sqrt{\pi D_{11}t^3}}e^{-\frac{((2n+1) L-x)^2}{4D_{11}t}}\\
	&=-\frac{e^{\alpha x+D_{11}s_0t}}{2\sqrt{\pi D_{11}t^3}}\sum_{n=0}^{\infty} \Bigg[(x+2n L)e^{-\frac{(x+2n L)^2}{4D_{11}t}}-((2n+1)L-x) e^{-\frac{((2n+1) L-x)^2}{4D_{11}t}}\Bigg]
\end{align*}
We can make the second sum more explicit:
\begin{align*}
	\sum_{n=0}^{\infty} ((2n+1)L-x) e^{-\frac{((2n+1) L-x)^2}{4D_{11}t}}&=\sum_{n=0}^{\infty}q(2(n+1)L-x), \qquad q(y)=ye^{-\frac{y^2}{4D_{11}t}}\\
	&=\sum_{n=1}^{\infty}q(2nL-x)\\
	&=-\sum_{n=1}^{\infty}q(x-2nL), \qquad \text{since}\ \ q\  \text{is odd}\\
	&=-\sum^{-1}_{n=-\infty}q(x+2nL)\\
	&=-\sum^{-1}_{n=-\infty}(x+2nL)e^{-\frac{(x+2nL)^2}{4D_{11}t}}
\end{align*}
Thus $g(x,t)$ takes the form
\begin{align}
	\nonumber
	g(x,t)&=-\frac{e^{\alpha x+D_{11}s_0t}}{2\sqrt{\pi D_{11}t^3}}\Bigg[\sum_{n=0}^{\infty} (x+2n L)e^{-\frac{(x+2n L)^2}{4D_{11}t}}+\sum^{-1}_{n=-\infty}(x+2nL)e^{-\frac{(x+2nL)^2}{4D_{11}t}}\Bigg]\\
	\nonumber
	\qquad
	\\
	\label{forme_g}
	g(x,t)&=-\frac{e^{\alpha x+D_{11}s_0t}}{2\sqrt{\pi D_{11}t^3}}\sum_{n\in \mathbb{Z}} (x+2n L)e^{-\frac{(x+2n L)^2}{4D_{11}t}}
\end{align}
Recall the following theorem \cite{katznelson2004introduction}.
\begin{theorem}[Poisson summation formula]
	\label{theo_somm}
	Let $a$ be a strictly positive real number and $\omega_{0} = \frac{2\pi}{a}$.
	If $f$ is a continuous function from $\mathbb{R}$ to $\mathbb{C}$, integrable, and such that
	$$
	\exists C > 0, \ \exists \alpha > 1, \ \forall x \in \mathbb{R}, \quad 
	|f(x)| \leq \frac{C}{\left(1 + |x|\right)^{\alpha}}
	$$
	and
	$$
	\sum_{m=-\infty}^{\infty} \left| \widehat{f}(m\omega_{0}) \right| < \infty,
	$$
	then
	$$
	\sum_{n=-\infty}^{\infty} f(t + na) 
	= \frac{1}{a} \sum_{m=-\infty}^{\infty} \widehat{f}(m\omega_{0}) \, e^{\, i m \omega_{0} t}.
	$$
\end{theorem}

The function $q(y)=ye^{-\frac{y^2}{4D_{11}t}}$ satisfies all the hypotheses of Theorem~\eqref{theo_somm}, so
\begin{align}
	\nonumber
	\sum_{n\in \mathbb{Z}} q(x+2n L)&=\frac{1}{2L}\sum_{n\in \mathbb{Z}} \hat{q}(\frac{n\pi}{L})e^{-\frac{i n\pi}{L}t},\qquad q(y)=ye^{-\frac{y^2}{4D_{11}t}}    \\
	\nonumber
	&=-\frac{(2\pi D_{11}t)\sqrt{4\pi D_{11}t}}{2L^2}\sum_{m\in \mathbb{Z}} im e^{-D_{11}t\frac{m^2\pi^2}{L^2}}e^{\frac{i m\pi}{L}x}\\
	\nonumber
	&=-\frac{(2\pi D_{11}t)\sqrt{4\pi D_{11}t}}{2L^2}\sum_{m=1}^{\infty} im e^{-D_{11}t\frac{m^2\pi^2}{L^2}}\left(e^{-\frac{i m\pi}{L}x}-e^{\frac{i m\pi}{L}x}\right)
\end{align}
It follows that
\begin{align}
	\label{serie_q}
	\sum_{n\in \mathbb{Z}} q(x+2n L)&=\frac{4(\pi D_{11}t)^{\frac{3}{2}}}{L^2}\sum_{m=1}^{\infty} m e^{-D_{11}t\frac{m^2\pi^2}{L^2}}\sin \frac{m\pi x}{L}
\end{align}
Substituting relation \eqref{serie_q} into \eqref{forme_g},

\begin{align}
\label{serie_g}
	g(x,t)
    =-\frac{e^{\alpha x+D_{11}s_0t}}{2\sqrt{\pi D_{11}t^3}}\sum_{n\in \mathbb{Z}} q(x+2n L),\qquad q(y)=ye^{-\frac{y^2}{4D_{11}t}}
\end{align}
we obtain, equivalently to the residue method \eqref{residu_g}, the same expression for $g(x,t)$:

\begin{align}
	\label{serie_poisson_g}
	g(x,t)=-e^{\alpha x}\sum_{m=1}^{\infty}\frac{2m\pi D_{11}}{L^2}e^{s_mt}\sin (\frac{m\pi x}{L})
\end{align}

	The series \eqref{serie_g} and \eqref{serie_poisson_g} describing the solution \( c(x,t) \) describe the same convection-diffusion physics. This reflects two complementary mathematical viewpoints: complex analysis (the residue method leading to \eqref{serie_g}) and Fourier analysis (Poisson summation leading to \eqref{serie_poisson_g}). Their agreement, despite the different derivations, is consistent with the correctness of the computation. The residue method makes the spectral structure explicit through the poles of the Green's function, while the Poisson summation gives a series form that is convenient for computation. Together, the two representations offer complementary theoretical and numerical perspectives on flow in porous media.

Under the additional assumption $c_p(t)=c_p^0$,
\begin{align}
	c(x,t)
	&=c_a^0+\left(c_a^0-c_p^0\right)\int_{0}^{t}g(x,\tau)d\tau
\end{align}

In this second part, we use the analytical solution modelling pollutant convection on the domain $[0,L]$ to obtain an extension to the half-space $[0,+\infty[$.
From the expression of $g$,
\begin{align*}
	g(x,t)&=-\frac{e^{\alpha x+D_{11}s_0t}}{2\sqrt{\pi D_{11}t^3}}\sum_{n\in \mathbb{Z}} (x+2n L)e^{-\frac{(x+2n L)^2}{4D_{11}t}}
\end{align*}
obtained from the series expansion, we now show, when $L\to \infty$, that the terms with $n\neq 0$ become negligible, since $e^{-\frac{(x+2n L)^2}{4D_{11}t}}$ decays exponentially with $L$.

\begin{lemma}[Negligibility of the $n\neq0$ terms]
	\label{lem:troncature}
	For fixed $x\in[0,+\infty[$ and $t>0$, set
	$$
	R_L(x,t)=\sum_{n\in\mathbb{Z}\setminus\{0\}} (x+2nL)\,e^{-\frac{(x+2nL)^2}{4D_{11}t}}.
	$$
	Then $R_L(x,t)\xrightarrow[L\to\infty]{}0$. Moreover, there exist $L_0>0$ and a constant $C>0$, independent of $L$, such that
	$$
	|R_L(x,t)|\leq C\,L\,e^{-\frac{L^2}{4D_{11}t}}, \qquad \forall L\geq L_0.
	$$
\end{lemma}

\begin{proof}
	For $L>x$ (which holds for $L$ large enough, $x$ being fixed), we have, for all $n\geq 1$,
	$$
	x+2nL\geq 2nL-x\geq (2n-1)L>0,
	$$
	and by symmetry, for all $n\leq -1$,
	$$
	|x+2nL|\geq (2|n|-1)L.
	$$
	The function $h(y)=y\,e^{-\frac{y^2}{4D_{11}t}}$ is decreasing on $[\sqrt{2D_{11}t},+\infty[$. As soon as $L\geq \sqrt{2D_{11}t}$, we thus have, for all $n\neq 0$,
	$$
	\big|(x+2nL)e^{-\frac{(x+2nL)^2}{4D_{11}t}}\big|\leq (2|n|-1)L\, e^{-\frac{((2|n|-1)L)^2}{4D_{11}t}}.
	$$
	Consequently,
	$$
	|R_L(x,t)|\leq 2\sum_{n=1}^{\infty} (2n-1)L\, e^{-\frac{((2n-1)L)^2}{4D_{11}t}}.
	$$
	Set $r=e^{-\frac{L^2}{4D_{11}t}}\in(0,1)$ for $L>0$. The general term of the series satisfies
	$$
	(2n-1)L\, e^{-\frac{((2n-1)L)^2}{4D_{11}t}}=(2n-1)L\,r^{(2n-1)^2}\leq (2n-1)L\,r^{2n-1},
	$$
	since $r\in(0,1)$ and $(2n-1)^2\geq 2n-1$. The series $\sum_{n\geq1}(2n-1)r^{2n-1}$ converges (power series in $r$ with radius of convergence $1$), bounded above by a universal constant $K$ as soon as $r\leq 1/2$, i.e. as soon as $L\geq L_0:=2\sqrt{D_{11}t\ln 2}$. We thus obtain
	$$
	|R_L(x,t)|\leq 2KL\,r=2KL\,e^{-\frac{L^2}{4D_{11}t}}=:C\,L\,e^{-\frac{L^2}{4D_{11}t}}, \qquad \forall L\geq L_0,
	$$
	which tends to $0$ as $L\to\infty$ and gives the announced uniform bound.
\end{proof}

\begin{proposition}[Limit of the Green's function]
	\label{pro:limite_g}
	For fixed $x\in[0,+\infty[$ and $t>0$,
	$$
	g(x,t)\xrightarrow[L\to\infty]{}g_\infty(x,t)=-\frac{x}{2\sqrt{\pi D_{11}t^3}}\,e^{-\frac{(x-v_0t)^2}{4D_{11}t}}.
	$$
	Moreover, for $L\geq L_0$, $g$ is bounded, uniformly in $L$, by a function that is integrable in $t$ on any interval $(0,T]$.
\end{proposition}

\begin{proof}
	From \eqref{forme_g},
	$$
	g(x,t)=-\frac{e^{\alpha x+D_{11}s_0t}}{2\sqrt{\pi D_{11}t^3}}\Big[x\,e^{-\frac{x^2}{4D_{11}t}}+R_L(x,t)\Big].
	$$
	Lemma \ref{lem:troncature} gives $R_L(x,t)\to 0$ as $L\to\infty$, so
	$$
	g(x,t)\xrightarrow[L\to\infty]{}-\frac{e^{\alpha x+D_{11}s_0t}}{2\sqrt{\pi D_{11}t^3}}\,x\,e^{-\frac{x^2}{4D_{11}t}}.
	$$
	Using $\alpha=\frac{v_0}{2D_{11}}$, $s_0=-\frac{v_0^2}{4D_{11}^2}$, and the identity
	$$
	\alpha x+D_{11}s_0t-\frac{x^2}{4D_{11}t}=-\frac{(x-v_0t)^2}{4D_{11}t}
	$$
	(already used to recover \eqref{residu_g} via residues), we obtain the stated formula. The uniform bound follows directly from the one in Lemma \ref{lem:troncature}, applied in the same identity.
\end{proof}

Proposition~\ref{pro:limite_g} only gives \emph{pointwise} convergence $g_L(x,\tau)\to g_\infty(x,\tau)$, for fixed $\tau$. To justify the limit $L\to\infty$ inside the convolution integral defining $c_L(x,t)$, a uniform (in $L$) integrable-in-$\tau$ bound is also needed; we establish it now.

\begin{lemma}[Integrability of the inverse Gaussian kernel]
	\label{lem:integrabilite_gauss}
	For every $a>0$, the function $\tau\mapsto \tau^{-3/2}e^{-a/\tau}$ is integrable on $(0,t]$, and
	$$
	\int_0^t \tau^{-3/2}e^{-a/\tau}\,d\tau \leq \sqrt{\frac{\pi}{a}}.
	$$
\end{lemma}

\begin{proof}
	Set $u=a/\tau$, so that $\tau=a/u$ and $d\tau=-\frac{a}{u^2}du$. As $\tau$ ranges over $(0,t]$, $u$ ranges over $[a/t,+\infty)$. We obtain
    \begin{align*}
        \int_0^t \tau^{-3/2}e^{-a/\tau}d\tau = \int_{a/t}^{\infty} \left(\frac{a}{u}\right)^{-3/2}e^{-u}\frac{a}{u^2}\,du & =a^{-1/2}\int_{a/t}^{\infty} u^{-1/2}e^{-u}\,du\\
        &\leq a^{-1/2}\int_0^\infty u^{-1/2}e^{-u}\,du \\
        &= a^{-1/2}\Gamma\!\left(\tfrac12\right)=\sqrt{\frac{\pi}{a}}<+\infty.
    \end{align*}
	
\end{proof}

\begin{lemma}[Boundedness of $c_p$]
	\label{lem:borne_cp}
	Under Assumption~\ref{H3}, there exists $M_{c_p}>0$ such that $|c_p(s)|\leq M_{c_p}$ for all $s\in[0,t]$.
\end{lemma}

\begin{proof}
	By Assumption~\ref{H3}, $c_p$ is of class $\mathcal{C}^1(\mathbb{R}^+)$, hence continuous on the compact set $[0,t]$; by the Weierstrass theorem, it is bounded there.
\end{proof}

\begin{corollary}[Uniform domination and limit in the convolution integral]
	\label{cor:passage_limite}
	For fixed $x\in[0,+\infty[$ and $t>0$,
	$$
	\lim_{L\to+\infty} c_L(x,t) = c_\infty(x,t),
	$$
	where $c_L(x,t)=c_a^0+\int_0^t g_L(x,\tau)\big(c_a^0-c_p(t-\tau)\big)d\tau$ and $c_\infty(x,t)=c_a^0+\int_0^t g_\infty(x,\tau)\big(c_a^0-c_p(t-\tau)\big)d\tau$.
\end{corollary}

\begin{proof}
	
\textbf{Case $x=0$.} As explained in Remark~\ref{rem:boundary_x0}, $g_L(0,\cdot)=-\delta_0=g_\infty(0,\cdot)$ as distributions, for every $L>0$ and in the limit $L\to\infty$ alike, since $G_L(0,s)=G_\infty(0,s)=-1$ for every $s$. 
    
	\textbf{Case $x>0$.} From the proof of Lemma~\ref{lem:troncature}, there exist $L_1=L_1(x,t)\geq \max\big(2\sqrt{D_{11}t\ln 2},\,\sqrt{2D_{11}t}\big)$ and $C>0$ such that, for every $L\geq L_1$ and every $\tau\in(0,t]$,
	$$
	|R_L(x,\tau)|\leq C L_1\, e^{-\frac{L_1^2}{4D_{11}\tau}}.
	$$
	Indeed, the threshold $L_0(\tau)=2\sqrt{D_{11}\tau\ln2}$ and the monotonicity threshold $\sqrt{2D_{11}\tau}$ used in the proof of Lemma~\ref{lem:troncature} are both increasing in $\tau$, hence bounded by their values at $\tau=t$; the bound obtained for fixed $\tau$ therefore remains valid uniformly for every $\tau\in(0,t]$ as soon as $L\geq L_1$.
	
	Since $D_{11}s_0\leq 0$ (because $s_0=-v_0^2/(4D_{11}^2)\leq 0$), we have $e^{D_{11}s_0\tau}\leq 1$ for every $\tau\geq0$. We deduce, for every $L\geq L_1$ and every $\tau\in(0,t]$,
	\begin{multline}
	    |g_L(x,\tau)| = \left|\frac{e^{\alpha x+D_{11}s_0\tau}}{2\sqrt{\pi D_{11}\tau^3}}\Big[xe^{-\frac{x^2}{4D_{11}\tau}}+R_L(x,\tau)\Big]\right|
	\leq \\
    \frac{e^{\alpha x}}{2\sqrt{\pi D_{11}}}\,\tau^{-3/2}\Big[xe^{-\frac{x^2}{4D_{11}\tau}}+CL_1e^{-\frac{L_1^2}{4D_{11}\tau}}\Big] =: M(x,\tau).
	\end{multline}

	By Lemma~\ref{lem:integrabilite_gauss}, applied successively with $a=\frac{x^2}{4D_{11}}$ and $a=\frac{L_1^2}{4D_{11}}$, each of the two terms $\tau^{-3/2}e^{-x^2/(4D_{11}\tau)}$ and $\tau^{-3/2}e^{-L_1^2/(4D_{11}\tau)}$ is integrable on $(0,t]$; hence $M(x,\cdot)\in L^1((0,t])$.
	
	Using Lemma~\ref{lem:borne_cp}, we obtain, for every $L\geq L_1$ and every $\tau\in(0,t]$,
	$$
	\Big|g_L(x,\tau)\big(c_a^0-c_p(t-\tau)\big)\Big| \leq \big(|c_a^0|+M_{c_p}\big)\,M(x,\tau) =: \widetilde{M}(\tau),
	$$
	with $\widetilde M\in L^1((0,t])$, independent of $L$.
	
	On the other hand, Proposition~\ref{pro:limite_g} gives $g_L(x,\tau)\to g_\infty(x,\tau)$ for every $\tau\in(0,t]$ as $L\to+\infty$, so
	$$
	g_L(x,\tau)\big(c_a^0-c_p(t-\tau)\big)\xrightarrow[L\to\infty]{}g_\infty(x,\tau)\big(c_a^0-c_p(t-\tau)\big) \qquad \text{for every } \tau\in(0,t].
	$$
	
	The hypotheses of the Lebesgue dominated convergence theorem are thus satisfied. We conclude
	$$
	\lim_{L\to+\infty}\int_0^t g_L(x,\tau)\big(c_a^0-c_p(t-\tau)\big)d\tau = \int_0^t g_\infty(x,\tau)\big(c_a^0-c_p(t-\tau)\big)d\tau,
	$$
	i.e. $\displaystyle\lim_{L\to+\infty}c_L(x,t)=c_\infty(x,t)$.
\end{proof}

Thus only the index-$0$ term contributes significantly, and Corollary~\ref{cor:passage_limite} justifies rigorously the approximation
\begin{align}
	g_{\infty}(x,t)= -\frac{e^{\alpha x+D_{11}s_0t}}{2\sqrt{\pi D_{11}t^3}} xe^{-\frac{x^2}{4D_{11}t}}=-\frac{x}{2\sqrt{\pi D_{11}t^3}} e^{-\frac{(x-v_0t)^2}{4D_{11}t}}
\end{align}
The solution $c(x,t)$ in the semi-infinite case is therefore
\begin{align}
	c_\infty(x,t)=c_a^0-\int_{0}^{t}\frac{x}{2\sqrt{\pi D_{11}\tau^3}} e^{-\frac{(x-v_0\tau)^2}{4D_{11}\tau}}\left(c_a^0-c_p(t-\tau)\right)d\tau
\end{align}
This expression gives the general solution for the semi-infinite case, where the influence of the boundary condition at $x=L$ disappears as $L\to \infty$. The function $g_\infty(x,t)$ is now a Gaussian shifted by the velocity $v_0$, and the solution depends on the boundary condition $c_p(t)$.

If $c_p(t)=c_p^0$ (constant), the integral can be evaluated explicitly in terms of error functions. The general solution then becomes

\begin{align}
	c_\infty(x,t)=c_a^0-\left(c_a^0-c_p^0\right)\int_{0}^{t}\frac{x}{2\sqrt{\pi D_{11}\tau^3}} e^{-\frac{(x-v_0\tau)^2}{4D_{11}\tau}}d\tau
\end{align}
To compute $I(x,t)=\int_0^tg_\infty(x,\tau)d\tau$, we first determine its Laplace transform:
\begin{align*}
	\mathcal{L}\Big[I(x,t)\Big]&=\int_0^\infty e^{-st}\left(\int_{0}^{t}\frac{x}{2\sqrt{\pi D_{11}\tau^3}} e^{-\frac{(x-v_0\tau)^2}{4D_{11}\tau}}d\tau\right)dt, \qquad 0\leq \tau \leq t\leq \infty
\end{align*}
Interchanging the integrals (Fubini's theorem~\cite{rudin1987real}),
\begin{align*}
	\mathcal{L}\Big[I(x,t)\Big]&=\int_{0}^{\infty}\frac{x}{2\sqrt{\pi D_{11}\tau^3}} e^{-\frac{(x-v_0\tau)^2}{4D_{11}\tau}}\left( \int_\tau^\infty e^{-st}dt \right)d\tau, \qquad 0\leq \tau \leq t\leq \infty\\
	&=\frac{x}{2s\sqrt{\pi D_{11}}}\int_{0}^{\infty}\frac{1}{\sqrt{\tau^3}} e^{-s\tau-\frac{(x-v_0\tau)^2}{4D_{11}\tau}}d\tau\\
	&=\frac{xe^{\frac{v_0x}{2D_{11}}}}{2s\sqrt{\pi D_{11}}}\int_{0}^{\infty}\tau^{-\frac{1}{2}-1} e^{-(s-\frac{v_0^2}{4D_{11}})\tau-\frac{x^2}{4D_{11}\tau}}d\tau\\
	&=\frac{xe^{\frac{v_0x}{2D_{11}}}}{2s\sqrt{\pi D_{11}}}\sqrt{\frac{4\pi D_{11}}{x^2}} \exp\left(-x\sqrt{\frac{s}{D_{11}}+\frac{v_0^2}{4D_{11}}}\right)\\
	&=e^{\frac{v_0 x}{2D_{11}}}\frac{1}{s}  \exp\left(-x\sqrt{\frac{s}{D_{11}}+\frac{v_0^2}{4D_{11}}}\right)
\end{align*}
Inverting $\mathcal{L}\Big[I(x,t)\Big]$, we obtain
$$
I(x,t)=\frac{e^{\frac{v_0 x}{2D_{11}}}}{2}\Bigg[e^{\frac{v_0 x}{2D_{11}}}\textbf{erfc}\left(\frac{x+v_0t}{2\sqrt{D_{11}t}}\right)+e^{-\frac{v_0 x}{2D_{11}}}\textbf{erfc}\left(\frac{x-v_0t}{2\sqrt{D_{11}t}}\right)\Bigg],
$$
with \textbf{erfc} the complementary error function,
$$
\textbf{erfc}\left(x\right)=\frac{2}{\sqrt{\pi}}\int_x^\infty e^{-t^2}
$$
Substituting $I(x,t)$ into the expression of $c(x,t)$, we obtain
\begin{equation}
	c_\infty(x,t)=c_a^0-\frac{c_a^0-c_p^0}{2}\Bigg[e^{\frac{v_0 x}{D_{11}}}\textbf{erfc}\left(\frac{x+v_0t}{2\sqrt{D_{11}t}}\right)+\textbf{erfc}\left(\frac{x-v_0t}{2\sqrt{D_{11}t}}\right)\Bigg]
\end{equation}
When $c_a(x)=0$ for all $x>0$, the solution reduces to
\begin{equation}
	\label{agio_banks}
	c_\infty(x,t)=\frac{c_p^0}{2}\Bigg[e^{\frac{v_0 x}{D_{11}}}\textbf{erfc}\left(\frac{x+v_0t}{2\sqrt{D_{11}t}}\right)+\textbf{erfc}\left(\frac{x-v_0t}{2\sqrt{D_{11}t}}\right)\Bigg]
\end{equation}
%\begin{figure}[htb]
%	\centering
%	\includegraphics[width=\linewidth, height=6cm, keepaspectratio]{./concentration_infini.png}
%	\caption{Concentration $c_\infty(x,t)$ in the one-directional half-space $\mathbb{R}_+$.}
%	\label{fig:concentration_infini}
%\end{figure}

	Solution \eqref{solution} extends the original formulation \eqref{agio_banks} for the longitudinal dispersion equation in porous media exactly like in reference~\cite{ogata1961solution}, combining molecular diffusion and convective transport of contaminants in a unidirectional flow. This generalization is consistent with the boundary conditions and reproduces the expected asymptotic behaviour, which supports its use for describing the coupled dispersion-convection process in spatially bounded domains.

\section{Conclusion}
\label{sec_conclusion}

We have presented a unified analytical solution of the 1D convection-diffusion problem at constant velocity, first on a bounded domain and then on the half-space. The Laplace transform gives two equivalent representations of the Green's function: a spectral series (residues) and a spatial series (Poisson summation), whose agreement supports the correctness of the solution. 
The lack of a complete analytical solution in the existing literature for the bounded case necessitated a comparison with the results reported in~\cite{van1982analytical}.The excellent agreement observed between our results and those of~\cite{van1982analytical} provides strong validation of our approach and enables us to draw further conclusions regarding the boundary conditions.
The limit $L\to\infty$, often taken for granted heuristically in the literature, is here derived, recovering in this way the classical Ogata and Banks type solution \cite{ogata1961solution}. A natural extension of this work concerns the case of a spatial and time-dependent wind velocity $V(x,t)$ and its generalization to two or three dimensions.

\medskip

\section*{Conflict of interest}
The authors declare no conflict of interest.

\vspace{5mm}

%\bibliography

\appendix
\section{The function $q$ satisfies the hypotheses of Theorem \ref{theo_somm}}

Provided the hypotheses of Theorem \ref{theo_somm} are satisfied, it allows us to express the solution $g$ as a mixed series. To check these hypotheses, consider the following representation of $g(x,t)$:
\begin{align}
	\label{serie_g_app}
	g(x,t)&=-\frac{e^{\alpha x+D_{11}s_0t}}{2\sqrt{\pi D_{11}t^3}}\sum_{n\in \mathbb{Z}} q(x+2n L),\qquad q(y)=ye^{-\frac{y^2}{4D_{11}t}}
\end{align}
\begin{enumerate}
	\item $q$ is integrable.\\
	Indeed,
	$$
	\int_{\mathbb{R}} |q(y)|dy=2\int_0^{+\infty} ye^{-\frac{y^2}{4D_{11}t}}dy=-4D_{11}t\Big[e^{-\frac{y^2}{4D_{11}t}}\Big]_0^{+\infty}=4D_{11}t<+\infty.
	$$
	Hence $q\in L^1(\mathbb{R})$.
	
	Let $r:=|y|$. We want to bound the function
	$$
	r(1+r)^\alpha e^{-\frac{r^2}{4D_{11}t}}
	$$
	uniformly in $r\geq 0$.
	\begin{itemize}
		\item If $r\leq 1$,
		$$
		r \leq 1 \implies  q(r)\leq 2^\alpha e^{-\frac{r^2}{4D_{11}t}}\leq 2^\alpha.
		$$
		\item If $r\geq 1$,
		\begin{align*}
			r \geq 1 \implies r(1+r)^\alpha\leq 2^\alpha r^{\alpha+1} e^{-\frac{r^2}{4D_{11}t}}
		\end{align*}
		Note that
		$$
		\sup_{r\geq 0} r^{p} e^{-\frac{r^2}{4D_{11}t}}=\left(\frac{2D_{11}tp}{e}\right)^{\frac{p}{2}}.
		$$
		Indeed, $h_p'(r_0)=0\implies r_0=\sqrt{2D_{11}tp}$ with $h_p(r)=r^{p} e^{-\frac{r^2}{4D_{11}t}}$.
		This leads to
		$$
		\sup_{r\geq 0} h_{p+1}(r)=\left(\frac{2D_{11}t(p+1)}{e}\right)^{\frac{p+1}{2}}.
		$$
		Hence, for every $r\geq 0$,
		$$
		q(r)\leq 2^\alpha \max\left\{1, \left(\frac{2D_{11}t(p+1)}{e}\right)^{\frac{p+1}{2}}\right\}=C_{D_{11}}.
		$$
	\end{itemize}
	\item $q$ satisfies the second hypothesis of Theorem \ref{theo_somm}.
	
	Since $q$ is integrable, let $\hat{q}$ denote its Fourier transform. A basic property of the Fourier transform gives
	\begin{align*}
		\hat{q}(\xi)&=i\hat{q}'_1(\xi), \qquad q_1(y)=e^{-\frac{y^2}{4D_{11}t}}\\
		&=i\frac{d}{d\xi} \sqrt{4\pi D_{11}t}e^{-D_{11}t\xi^2}\\
		&=-i(2D_{11}t)\sqrt{4\pi D_{11}t}\xi e^{-D_{11}t\xi^2}\\
		|\hat{q}(\xi)| &\leq (2D_{11}t)\sqrt{4\pi D_{11}t}|\xi|e^{-D_{11}t\xi^2}
	\end{align*}
	In particular, setting $\xi=mw_0$,
	the series
	$$
	\sum_{m=-\infty}^{\infty} \left| \hat{q}(m\omega_{0}) \right| < \infty,
	$$
	\noindent since
	$$
	\sum_{m=-\infty}^{\infty}\left| \hat{q}(m\omega_{0}) \right| \leq (2D_{11}t)\sqrt{4\pi D_{11}t}\sum_{m=-\infty}^{\infty} m|w_0|e^{-D_{11}t(mw_0)^2}
	$$
	\noindent and the series $\sum_{m=-\infty}^{\infty} m|w_0|e^{-s(mw_0)^2}$ converges for every $s>0$.
\end{enumerate}


\begin{thebibliography}{9}

\bibitem{abramowitz1948handbook}
{\small {\sc M. Abramowitz and I.A. Stegun}}: {\it Handbook of mathematical functions with formulas, graphs, and mathematical tables}, vol.~55, US Government printing office, 1948.

\bibitem{augustin2011assessment}
{\small {\sc M. Augustin, A. Caiazzo, A. Fiebach, J. Fuhrmann, V. John, A. Linke and R. Umla}}: {\it An assessment of discretizations for convection‑dominated convection‑diffusion equations}, Comput. Methods Appl. Mech. Engrg., {\bf 200}(47‑48) (2011), 3395--3409.

\bibitem{bermudez1987methode}
{\small {\sc A. Bermúdez and J. Durany}}: {\it La méthode des caractéristiques pour les problemes de convection‑diffusion stationnaires}, ESAIM Math. Model. Numer. Anal., {\bf 21}(1) (1987), 7--26.

\bibitem{deng2014integral}
{\small {\sc B. Deng, J. Li, B. Zhang and N. Li}}: {\it Integral transform solution for solute transport in multi‑layered porous media with the implicit treatment of the interface conditions and arbitrary boundary conditions}, J. Hydrol., {\bf 517} (2014), 566--573.

\bibitem{fahs2014new}
{\small {\sc M. Fahs, A. Younes and T.A. Mara}}: {\it A new benchmark semi‑analytical solution for density‑driven flow in porous media}, Adv. Water Resour., {\bf 70} (2014), 24--35.

\bibitem{katznelson2004introduction}
{\small {\sc Y. Katznelson}}: {\it An introduction to harmonic analysis}, Cambridge University Press, 2004.

\bibitem{kumar2010analytical}
{\small {\sc A. Kumar, D.K. Jaiswal and N. Kumar}}: {\it Analytical solutions to one‑dimensional advection‑diffusion equation with variable coefficients in semi‑infinite media}, J. Hydrol., {\bf 380}(3‑4) (2010), 330--337.

\bibitem{kunasegaran2025analytical}
{\small {\sc P. Kunasegaran and Z.M. Isa}}: {\it Analytical solution of two‑dimensional fractional advection‑diffusion equation with instantaneous source and time‑varying coefficients}, Gulf J. Math., {\bf 21}(1) (2025), 470--485.

\bibitem{linss2003layer}
{\small {\sc T. Linß}}: {\it Layer‑adapted meshes for convection‑diffusion problems}, Comput. Methods Appl. Mech. Engrg., {\bf 192}(9‑10) (2003), 1061--1105.

\bibitem{ogata1961solution}
{\small {\sc A. Ogata and R.B. Banks}}: {\it A solution of the differential equation of longitudinal dispersion in porous media: fluid movement in earth materials}, US Government Printing Office, 1961.

\bibitem{rudin1987real}
{\small {\sc W. Rudin}}: {\it Real and complex analysis}, McGraw-Hill, Inc., 1987.

\bibitem{singh2015scale}
{\small {\sc M.K. Singh and P. Das}}: {\it Scale dependent solute dispersion with linear isotherm in heterogeneous medium}, J. Hydrol., {\bf 520} (2015), 289--299.

\bibitem{stockie2011mathematics}
{\small {\sc J.M. Stockie}}: {\it The mathematics of atmospheric dispersion modeling}, SIAM Rev., {\bf 53}(2) (2011), 349--372.

\bibitem{stynes2005steady}
{\small {\sc M. Stynes}}: {\it Steady‑state convection‑diffusion problems}, Acta Numer., {\bf 14} (2005), 445--508.

\bibitem{szymczak2003boundary}
{\small {\sc P. Szymczak and A.J.C. Ladd}}: {\it Boundary conditions for stochastic solutions of the convection‑diffusion equation}, Phys. Rev. E, {\bf 68}(3) (2003), 036704.

\bibitem{van1982analytical}
{\small {\sc M.Th. van Genuchten and W.J. Alves}}: {\it Analytical solutions of the one-dimensional convective-dispersive solute transport equation}, U.S. Department of Agriculture, Technical Bulletin No.~1661, 1982.

\bibitem{zoppou1999analytical}
{\small {\sc C. Zoppou and J.H. Knight}}: {\it Analytical solution of a spatially variable coefficient advection‑diffusion equation in up to three dimensions}, Appl. Math. Model., {\bf 23}(9) (1999), 667--685.

\end{thebibliography}
\end{document}